\documentclass[12pt,english,a4paper]{smfart}

\usepackage[T1]{fontenc}
\usepackage{lmodern}
\usepackage{smfthm}
\usepackage{smfenum}
\usepackage{amssymb,amscd}
\usepackage{braket}
\usepackage[all]{xy}
\usepackage{comment}
\usepackage{physics}
\usepackage{mathtools}

\theoremstyle{plain}
\newtheorem{Theo}{Theorem}[subsection]
\newtheorem*{Theo*}{Theorem}
\newtheorem{Defi}[Theo]{Definition}
\newtheorem{Lemm}[Theo]{Lemma}
\newtheorem{Prop}[Theo]{Proposition}
\newtheorem{Coro}[Theo]{Corollary}
\newtheorem{Conj}[Theo]{Conjecture}

\theoremstyle{remark}
\newtheorem{Rema}[Theo]{Remark}

\newcommand{\Qp}{{\mathbf{Q}_p}}
\newcommand{\Zp}{{\mathbf{Z}_p}}
\newcommand{\Z}{\mathbf{Z}}
\newcommand{\Q}{\mathbf{Q}}
\newcommand{\dist}{{\calR^+_L(\Gamma)}}

\renewcommand{\phi}{\varphi}
\renewcommand{\geq}{\geqslant}
\renewcommand{\leq}{\leqslant} 
\newcommand{\calR}{\mathcal{R}}
\newcommand{\Gal}{\mathrm{Gal}}
\newcommand{\Hom}{\mathrm{Hom}}
\newcommand{\GL}{\mathrm{GL}}

\newcommand{\cris}{{\operatorname{cris}}}

\newcommand{\dR}{{\operatorname{dR}}}
\newcommand{\rig}{{\operatorname{rig}}}
\newcommand{\dif}{{\operatorname{dif}}}
\newcommand{\Iw}{{\operatorname{{Iw}}}}
\newcommand{\tors}{{\operatorname{tors}}}
\newcommand{\free}{{\operatorname{free}}}

\newcommand{\Hpg}{\mathrm{H}_{\phi, \gamma}}
\newcommand{\Hsg}{\mathrm{H}_{\psi, \gamma}}
\newcommand{\Hpsi}{\mathrm{H}_{\psi}}
\newcommand{\Hg}{\mathrm{H}_{\gamma}}
\newcommand{\pg}{(\phi, \Gamma)}

\newcommand{\dcris}{\mathbf{D}_{\cris}}

\newcommand{\ddr}{\mathbf{D}_{\dR}}
\newcommand{\ddif}{\mathbf{D}_{\dif}}
\newcommand{\ddifn}{\mathbf{D}_{\dif,n}}
\newcommand{\ddifm}{\mathbf{D}_{\dif,m}}
\newcommand{\nrig}{\mathbf{N}_{\rig}}

\newcommand{\drig}[1]{\mathbf{D}^{\dagger #1}_{\mathrm{rig}}}

\newcommand{\Exp}{\mathrm{Exp}}
\newcommand{\id}{{\mathrm{id}}}
\newcommand{\can}{{\mathrm{can}}}

\newcommand{\isom}{\xrightarrow{\sim}}
\renewcommand{\ev}{{\mathrm{ev}}}
\newcommand{\incl}{{\mathrm{incl}}}

\newcommand{\chara}{{\mathrm{char}}}
\newcommand{\Det}{\mathrm{Det}}
\renewcommand{\det}{\mathrm{det}}

\author{Tetsuya Ishida and Kentaro Nakamura}

\title[Local $\varepsilon$-conjecture and $p$-adic differential equations]
{Local $\varepsilon$-conjecture and $p$-adic differential equations}



\begin{document}

\begin{abstract}
Laurent Berger attached a $p$-adic differential equation $\nrig(M)$ with a Frobenius structure to an arbitrary de Rham $\pg$-module $M$ over a Robba ring. 
In this article, we compare the local epsilon conjecture for the cyclotomic deformation of $M$ with that of $\nrig(M)$. 
We first define an isomorphism between the fundamental lines of their cyclotomic deformations using the second author's results on the big exponential map. 
As a main result of the article, we show that this isomorphism enables us to reduce the local epsilon conjecture for the cyclotomic deformation of $M$ to that of $\nrig(M)$.

\end{abstract}

\maketitle

\tableofcontents

\setlength{\baselineskip}{18pt}

\section{Introduction/Notation}
\label{1}

In \cite{Kat93a}, Kato formulated a conjecture called the generalized Iwasawa main conjecture, which is a vast generalization of the Iwasawa main conjecture and Bloch-Kato conjecture.
It claims the existence of so-called zeta isomorphisms for any family of $p$-adic Galois representations of $G_\Q$, interpolating the zeta elements of geometric $p$-adic Galois representations.
Note that a similar conjecture was formulated by Fontaine and Perrin-Riou in \cite{FP94}.
Since the zeta elements are conjectural bases in (the determinants of) the Galois cohomologies and closely related to the $L$-functions, it is natural to regard the zeta isomorphisms as algebraic counterparts of the $L$-functions.
In \cite{Kat93b} and \cite{FK06}, Kato's local and global $\varepsilon$-conjectures are formulated as algebraic analogue of the functional equations of $L$-functions; the local $\varepsilon$-conjecture claims the existence of the local $\varepsilon$-isomorphisms, the algebraic analogue of local $\varepsilon$-factors for families of $p$-adic representations of $G_{\Q_l}$, and the global $\varepsilon$-conjecture states that the zeta isomorphisms satisfies the functional equations whose local factors are the local $\varepsilon$-isomorphisms.

The local $\varepsilon$-conjecture for $l \neq p$ is proved \cite{Yas09}, \cite{Kak14}.
But for the case $l = p$, which we treat in this paper, the existence of the local $\varepsilon$-isomorphisms are proved for limited families and the conjecture is still open.
In particular, by generalizing the conjecture for $\pg$-modules over relative Robba rings, the second author proves the existence of $\varepsilon$-isomorphisms for trianguline representations.
The conjecture has turned out to be closely related to the Coleman isomorphisms \cite{Kat93b} \cite{Ven13}, the Perrin-Riou maps \cite{BB08} \cite{LVZ13}, and also the $p$-adic local Langlands correspondence \cite{Nak17b} \cite{RJ18}.

Our main theorem compares the local $\varepsilon$-isomorphisms of the following different objects.
Let $M$ be an arbitrary de Rham $\pg$-module over a Robba ring.
The first object is the cyclotomic deformation of $M$.
The second one is the cyclotomic deformation of $\nrig(M)$, where $\nrig (M)$ is the $p$-adic differential equation attached to $M$ by Laurent Berger.
We remark that the existences of their local $\varepsilon$-isomorphisms are still conjectural.
The main theorem claims that the difference of their local $\varepsilon$-isomorphisms is written as the generalized Perrin-Riou map defined by the second author in \cite{Nak14}.

To make the statement of the main theorem more precise, we recall $\pg$-modules over Robba rings and the local $\varepsilon$-conjecture for them.

A $\pg$-module $D$ is a module equipped with a suitable endomorphism $\phi: D \to D$ and a continuous group action of $\Gamma=\Gal(\Qp(\mu_{p^\infty})/\Qp)$, where $\mu_{p^\infty}$ is the group of $p$-power roots of unity in $\overline{\Q}_p$.
There are several specific rings over which $\pg$-modules are useful to study $p$-adic representations.
An important case is the Robba rings $\calR_L$ with their coefficients in local fields $L$; by results of Fontaine \cite{Fon90}, Cherbonnier and Colmez \cite{CC99} and Kedlaya \cite{Ked08}, the category of $p$-adic representations over $L$ can be embedded fully and faithfully into the one of $\pg$-modules over $\calR_L$.
A lot of important notions of $p$-adic Hodge theory can be generalized to $\pg$-modules over $\calR_L$, such as the functors $\dcris$ and $\ddr$ \cite{Ber02}, or Bloch-Kato's exponential maps \cite{Ber03}, \cite{Nak14}.
Another important feature is that, when a $\pg$-module $M$ is de Rham, Berger attached to $M$ a $p$-adic differential equation $\nrig(M)$ with Frobenius structure; as its application, one can prove the $p$-adic monodromy theorem for $p$-adic representations by reducing it to that for $p$-adic differential equations, or Colmez-Fontaine's theorem \cite{Ber02}, \cite{Ber08}.

In \cite{Nak17a}, the second author formulated the local $\varepsilon$-conjecture for $\pg$-modules over relative Robba rings, generalizing the Kato's conjecture for $p$-adic representations.
We recall only the conjecture for the cyclotomic deformations of de Rham $\pg$-modules, since it is the case we treat in this paper.
Let $L$ be a finite extension of $\Qp$, and $M$ be a $\pg$-module over the Robba ring $\calR_L$ with coefficients in $L$.
Then, one can attach to $M$ a (graded) invertible module $\Delta_L(M)$ over $L$ and $\Delta^\Iw_L(M)$ over $\dist$ for a $\pg$-module $M$ over $\calR_L$, where we put $\calR^+_L(\Gamma) = \Gamma(\mathcal{W},\mathcal{O}_\mathcal{W})$ and $\mathcal{W}$ the Berthelot generic fiber of the Iwasawa algebra $\mathcal{O}_L[[\Gamma]]$.
When $M$ is de Rham, he constructed a canonical trivialization isomorphism
\[\varepsilon_L^{\dR}(M): L \xrightarrow{\sim} \Delta_L(M). \]
Its definition involves a lot of notions of $p$-adic Hodge theory, such as the theory of local constants ($\varepsilon$-constants and $L$-constants), Bloch-Kato's exponential and dual exponential maps, Hodge-Tate weights.
Then the local $\varepsilon$-conjecture in this situation claims that, there exists a unique isomorphism 
\[ \varepsilon^{\Iw}_L(M): \dist \xrightarrow{\sim} \Delta^\Iw_L(M)\]
interpolating $\varepsilon^\dR_L(M(\delta))$ for any de Rham character $\delta: \Gamma \to L^\times$, i.e.\ any character of the form $\delta = \chi^k \tilde{\delta}$ for $k \in \Z$ and a finite character $\tilde{\delta}$, where $\chi$ is the cyclotomic character.
More precisely, $\varepsilon^{\Iw}_L(M)$ is required to make the following diagram
\[
\xymatrixcolsep{4pc}
\xymatrix{
\Delta^\Iw_L(M) \otimes_{f_\delta} L\ar[r]^-{\ev_\delta}  \ar[d]_-{\varepsilon^\Iw_L(M) \otimes \id}  & \Delta_L(M(\delta)) \ar@{->}[d]^-{\varepsilon^{\mathrm{dR}}_L(M(\delta))} \\ 
\dist \otimes_{f_\delta} L \ar[r]_-{\can} & L
}
\]
commute for any de Rham character $\delta$ of $\Gamma$, where $f_\delta: \dist \to \dist$ is a continuous homomorphism of $L$-algebras given by $[g] \mapsto \delta(g)^{-1}$ and $\ev_\delta$ is a canonical isomorphism induced by the specialization a $f_\delta$.
In the original article of Kato \cite{Kat93b}, he predicts the conjectural base $\varepsilon^\Iw_{\mathcal{O}_L}(T)$ of an invertible $\mathcal{O}_L[[\Gamma]]$-module $\Delta^\Iw_{\mathcal{O}_L}(T)$ similarly defined for any $\mathcal{O}_L$-representation $T$ of $G_\Qp$.
In \cite{Nak17a}, the second author predicts the equality $ \varepsilon^\Iw_{\mathcal{O}_L}(T) \otimes \id = \varepsilon^\Iw_{L}(\drig{}(T[1/p]))$, that is, the right hand side has an integral structure in the \'etale case.

The following is the main theorem of this paper, which can be regarded as an extension of the studies in \cite{Nak14} and \cite{Nak17a}.
It roughly states that, for a general de Rham $\pg$-module $M$ over $\calR_L$ and the $p$-adic differential equation $\nrig(M)$ attached to $M$, the differences of $\varepsilon_L(M(\delta))$ and $\varepsilon_L(\nrig(M)(\delta))$ for the de Rham characters $\delta$ of $\Gamma$ are interpolated by the generalized Perrin-Riou map in \cite{Nak14}.

\begin{Theo*}
There exists an $\dist$-linear isomorphism
\[ \Exp(M): \Delta^\Iw_L(\nrig(M)) \xrightarrow{\sim} \Delta^\Iw_L(M) \]
whose specialization at any de Rham character $\delta$ of $\Gamma$ makes the following diagram
\[
\xymatrix{
& 1_L \ar[dl]_-{\varepsilon^{\dR}_L(\nrig(M)(\delta))}\ar[dr]^-{\varepsilon^{\dR}_L(M(\delta))}\ar@{}[d]|{} & \\
\Delta_{L}(\nrig(M)(\delta)) \ar[rr]_-{ \Exp(M)_\delta}   && \Delta_L(M(\delta))\\
}
\]
commute, where the isomorphism $\Exp(M)_\delta$ is defined by the following commutative diagram
\[
\xymatrix{
\Delta_{L}(\nrig(M)(\delta)) \ar[rr]^-{ \Exp(M)_\delta } &&  \Delta_{L}(M(\delta))\\
\ar[u]^-{\ev_\delta}  \Delta^\Iw_{L}(\nrig(M)) \otimes_{f_\delta} L \ar[rr]_-{\Exp(M) \otimes \id } && \Delta^\Iw_{L}(M) \otimes_{f_\delta} L \ar[u]_-{\ev_\delta}. \\
}
\]

In particular, if $\varepsilon^{\Iw}_L(\nrig(M))$ exists, then $\varepsilon^{\Iw}_L(M)$ also exists and is written as 
\[ \varepsilon^\Iw_L(M) = \Exp(M) \circ \varepsilon^{\Iw}_L(\nrig(M)). \]
\end{Theo*}

We remark that our theorem can be regarded as a refined interpolation formula for Bloch-Kato morphisms.
The isomorphism $\Exp(M)$ is obtained by the generalized Perrin-Riou's big exponential map
\[ \Exp_{M,h}: \Hpsi^1(\nrig(M)) \to \Hpsi^1(M) \] 
of \cite{Nak14} for de Rham $\pg$-module $M$, in conjunction with one of the main results, theorem $\delta(D)$.
The big exponential maps are first introduced by Perrin-Riou \cite{Per94} for crystalline representations and used essentially in her study of $p$-adic $L$-functions, and then generalized to de Rham representations \cite{Col98} and to de Rham $\pg$-modules \cite{Nak14}.
Their key feature is that they interpolate the Bloch-Kato's morphisms of twists $\exp_{M(\chi^k \tilde{\delta})}$ and $\exp^*_{M(\chi^k \tilde{\delta})}$ for suitable $k \in \Z$.
The theorem can be seen as a refinement of such interpolation formulae; our big exponential map $\Exp(M)$ interpolates, at any twists $\delta = \chi^k \tilde{\delta}$ for any $k \in \Z$ and $\tilde{\delta}$, not only the maps $\exp_{M(\delta)} $and $\exp^*_{M(\delta)}$ but also another exponential map $\exp_{f,M(\delta)}: \dcris(M(\delta)) \to \Hpg^1(M(\delta))$, which is closely related with the exceptional zeros for $p$-adic $L$-functions.
We note that, even when $M$ comes from a crystalline $p$-adic representation, the map $\exp_{f,M(\delta)}$ is non-zero in general and we can obtain its information via our refined formula.

We also remark a relation of our theorem to the local $\varepsilon$-conjecture itself.
The local $\varepsilon$-conjecture for the cyclotomic deformation of a general de Rham $\pg$-module is not proved yet, and only the following special cases are proved.
\begin{itemize}
	\item The case of rank $1$ Galois representations (i.e.\ rank $1$ \'etale $\pg$-modules) is proved by Kato in \cite{Kat93b} (proofs taking account of signs is given in \cite{FK06} briefly and in \cite{Ven13} in detail.)
	\item The case of crystalline representations is proved by Benois and Berger in \cite{BB08}, which is generalized by Loeffler, Venjakob, and Zerbes in \cite{LVZ13}, and by Bellovin and Venjakob in \cite{BV19}.
	\item The case of trianguline $\pg$-modules over relative Robba rings, including all semi-stable representations and also the representations associated to finite slope overconvergent modular forms, is proved by the second author in \cite{Nak17a}.
	\item The case of rank $2$ Galois representations is proved by the second author in \cite{Nak17b} in almost all cases and completed by Rodrigues Jacinto \cite{RJ18}, by showing its close relation to the $p$-adic local Langlands conjecture for $\GL_2(\Qp)$.
\end{itemize}
By the last assertion of the theorem, we can reduce the local $\varepsilon$-conjecture for the cyclotomic deformation of arbitrary de Rham $\pg$-module $M$ to that of $\nrig(M)$.
This reduction seems a useful approach, since $\nrig(M)$ is relatively simple (all of its Hodge-Tate weights are zero) and also has an additional structure of a $p$-adic differential equation with a Frobenius structure so that we can utilize the theory of $p$-adic differential equations.
We note that such a reduction is implicitly used to prove the trianguline case, and this theorem is stated as a conjecture \cite[Remark 4.15]{Nak17a}; see also Remark \ref{rank 1 and conjecture}.

The structure of the paper is as follows.
In section \ref{2}, we recall definitions about $\pg$-modules over Robba rings and prove the key lemma Lemma \ref{keylemma} on a relation of Bloch-Kato's morphisms and distributions.
In section \ref{3}, we recall (a special case of) the local $\varepsilon$-conjecture for $\pg$-modules studied in \cite{Nak17a}, introduce the $p$-adic differential equation $\nrig(M)$ for a de Rham $\pg$-module $M$, and construct our big exponential map $\Exp(M): \Delta^\Iw_L(\nrig(M)) \xrightarrow{\sim} \Delta^\Iw_L(M)$; it is induced by distribution, and the construction depends heavily on \cite{Nak14}.
In section \ref{4}, we state our main theorem and prove it, by introducing the notion of genericity, deducing the proof of the general case to the case of generic, and proving the generic case by applying the key lemma.

\textit{Notation.}
Let $p$ be a prime number.
We fix the algebraic closure $\overline{\Q}_p$ of the $p$-adic number field $\Q_p$.
Let $L$ be a finite extension of $\Q_p$.
Let $\mu_{p^\infty}$ denote the group of $p$-power roots of unity in $\overline{\Q}_p$.
We fix primitive $p^n$-th roots of unity $\zeta_{p^n} \in \mu_{p^\infty}$ such that $\zeta^p_{p^{n+1}} = \zeta_{p^n}$ for any $n \geq \Z_{\geq 1}.$
The set $\Gamma=\mathrm{Gal}(\Q_p(\mu_{p^\infty})/\Qp)$.
Let $\Delta \subseteq \Gamma$ be the $p$-torsion subgroup of $\Gamma$ and put $p_\Delta = \frac{1}{|\Delta|} \sum_{\sigma \in \Delta} \sigma$.
We fix an element $\gamma \in \Gamma$ whose image in $\Gamma/\Delta$ is a topological generator.
The cyclotomic character on $\Gamma$ is denoted by $\chi: \Gamma \xrightarrow{\sim} \Z_p^\times$, which is characterized by $\gamma(\zeta) = \zeta^{\chi(\gamma)}$ for all $\zeta \in \mu_{p^\infty}$ and $\gamma \in \Gamma$.
For a ring $R$, the objects of the category of graded invertible $R$-modules are written as the pairs $(\mathcal{L},r)$ of an invertible $R$-module $\mathcal{L}$ and a continuous function $r: \mathrm{Spec}(R) \to \Z $, and the product $\boxtimes$ is defined by 
$(\mathcal{L}_1, r_1)\boxtimes(\mathcal{L}_2, r_2):=(\mathcal{L}_1\otimes_R\mathcal{L}_2, r_1+r_2)$.
We put $1_R \coloneqq (R,0)$.

\section{Review of the theory of $\pg$-modules over Robba rings}
\label{2}

In this section, we first recall the definition of $\pg$-modules over Robba rings, their cohomologies, and some notions of $p$-adic Hodge theory.
Then, we study several kinds of morphisms defined by a distribution.
Theorem \ref{keylemma} is the key result, which describes a relation between such morphisms and Bloch-Kato's morphisms.

\subsection{$\pg$-modules over Robba rings}
\label{21}

For each integer $n \in \Z_{\geq 1}$, put 
\[ \calR_L^{(n)} = \set{ \sum_{i\in \Z} a_i T^i | a_i \in L, \sum_{i\in \Z} a_i T^i \text{ is convergent on } |\zeta_{p^n} -1| \leq |T| <1 }.\] 
We put $\calR_L = \cup_{n \geq 1} \calR_L^{(n)}$, with which we can equip a canonical LF-topology and we call the Robba ring over $L$.
Put $t = \log(1+T) \in \calR_L$.
There is an operator $\phi: \calR_L \to \calR_L$ and a group action of $\Gamma$ on $\calR_L$, both of which are continuous and linear over $L$ satisfying
\[ \phi(T) = (1+T)^p -1, \, \gamma(T) = (1+T)^{\chi(\gamma)} -1 \]
for any $\gamma \in \Gamma$.
A tuple $((1+T)^i)_{i=0, \dots, p-1}$ is a basis of $\calR_L$ over $\phi(\calR_L)$, and we can define a map $\psi:\calR_L \to \calR_L$ by 
\[ \psi \qty(\sum_{i=0}^{p-1} \phi(f_i) (1+T)^i) = f_0 \]
for $f_i \in \calR_L$.
Then $\psi$-operator turns out to be continuous and commutes with $\Gamma$.

For each $n \in \Z_{\geq 1}$, set $L_n = \Qp(\zeta_{p^n}) \otimes_\Qp L$.
Then one has a continuous $\Gamma$-equivariant homomorphism
\[\iota_n: \calR_L^{(n)} \to L_n[[t]] \]
of $L$-algebras such that
\[ \iota_n(T) = \zeta_{p^n} \exp \qty( \frac{t}{p^n}) -1, \]
which satisfies the following commutative diagram
\[
\xymatrixcolsep{4pc}
\xymatrix{
\calR^{(n)}_L \ar[r]^-{\iota_n}  \ar@{->}[d]_-{\phi}  & L_n[[t]] \ar@{->}[d]^-{\incl} \\
\calR^{(n+1)}_L \ar[r]_-{\iota_{n+1}} & L_{n+1}[[t]].
}
\]
\begin{Defi}
A $\pg$-module over $\calR_L$ is a free $\calR_L$-module $D$ of finite rank equipped with a semilinear endomorphism $\phi: D \to D$ over satisfying $\phi^* D = D$ and a continuous $\Gamma$-action commuting to $\phi$.
\end{Defi}

The following lemma is Theorem 1.3.3 of \cite{Ber08}.
\begin{Lemm}\label{increasing sequence of phi-module}
Let $D$ be a $\pg$-module over $\calR_L$.
Then, there exists an integer $n \geq 1$ such that there exists a unique $\Gamma$-stable $\calR^{(m)}_L$-submodule $D^{(m)} \subseteq D$ for each $m \geq n$ such that for any $m \geq n$ we have $\calR^{m+1}_L \otimes_{\calR^m_L,\phi} D^{(m)} = D^{(m+1)}$ and $\calR_L \otimes_{\calR^m_L} D^{(m)} = D$.
\end{Lemm}

The smallest integer $n$ satisfying the property in Lemma \ref{increasing sequence of phi-module} is denoted as $n(D)$.

For a $\pg$-module $D$ over $\calR_L$, one can define $\psi$-operator on $D$ by $\psi(\phi(x) \otimes f)= x \otimes \psi(f)$ for $x \in D$ and $f \in \calR_L$, which turns out to be well-defined, continuous and $L$-linear.

For each $n \geq n(D)$, define
\[ \ddifn^+(D) = D^{(n)} \otimes_{\iota_n,\calR_L^{(n)}} L_n[[t]]. \]
We put $\iota_n: D^{(n)} \to \ddifn^+(D): x \mapsto x \otimes 1$ and 
\[ \can=\can_n: \ddifn^+(D) \to \mathbf{D}_{\dif,n+1}^+(D): f(t) \otimes x \mapsto f(t) \cdot \iota_{n+1}(\phi(x)) \]
for $f(t) \in L_n[[t]]$ and $x \in D^{(n)}$.
We set 
\[ \ddifn(D) = \ddifn^+(D)[1/t], \ddif^{(+)}(D) = \varinjlim_{n} \ddifn^{(+)} (D),\]
here the injective limit is taken over $(\can_n)_{n \geq n(D)}$.
On these modules, we define $\Gamma$-actions diagonally.

One can consider several complexes which are related to $p$-adic Hodge theory.
Let $R$ be a topological ring, $X$ a topological module over $R$.
If $X$ is equipped with a continuous $R$-linear $\Gamma$-action, then we put as a complex of $R$-modules  
\[ C^{\bullet}_\gamma(X)\coloneqq [ X^\Delta \xrightarrow{\gamma -1} X^\Delta ] \]
concentrated in degree $[0,1]$.
If $X$ is furthermore equipped with a continuous $R$-linear action of $\phi$ or $\psi$ commuting the $\Gamma$-action, then we put, again as a complex of $R$-modules,
\[ C^{\bullet}_{*,\gamma}(X)\coloneqq [ X^\Delta \xrightarrow{(\gamma-1,*-1)} X^\Delta \oplus X^\Delta \xrightarrow{(*-1)\oplus(1-\gamma)} X^\Delta ] \]
concentrated in degree $[0,2]$ for $* = \phi, \psi$,  and 
\[ C^{\bullet}_{\psi}(X)\coloneqq [ X^\Delta \xrightarrow{\psi -1} X^\Delta ] \]
concentrated in degree $[1,2]$.
For each complex $C^\bullet_\square(X)$ above, its $i$-th cohomology group is denoted as $\mathrm{H}^i_{\square}$.
For a $\pg$-module $D$ over $\calR_L$ of rank $r$, we use the following special notations
\[ \dcris(D)= \Hg^0(D[1/t]) , \quad \ddr(D)=\Hg^0(\ddif(D)). \]
These spaces are of dimension $\leq r$ over $L$, and we say $D$ is crystalline (resp. de Rham) if $\dim_L(\dcris(D)) = r$ (resp. $\dim_L(\ddr(D))= r$.)
For $i \in \Z$, we also define $ \ddr^i(D) = \ddr(D) \cap t^i \ddif^+(D)$ and $t(D) = \ddr(D)/\ddr^0(D)$.
When $D$ is de Rham, then we say that $h \in \Z$ is a Hodge-Tate weight of $D$ if $\ddr^{-h}(D) /\ddr^{-h+1}(D) \neq 0$, and refer its dimension as the multiplicity of $h$.
We put $h_M$ as the sum of the Hodge-Tate weights of $M$ with multiplicity.

\subsection{Morphisms induced by distributions}
\label{22}

Let $\mathcal{W}$ be the Berthelot generic fiber of the Iwasawa algebra $\Zp[[\Gamma]]$, and define the distribution algebra $\dist$ as the global section $\Gamma(\mathcal{W},\mathcal{O})$.
In this subsection, we consider several morphisms induced by an element of $\dist$.
Then we prove a theorem about relationships between such morphisms and Bloch-Kato morphisms, which will be used as a key ingredient for our main theorem.

We recall natural $\dist$-actions on several objects related to a $\pg$-module $D$ over $\calR_L$.
For each $n \geq n(D)$, we can equip with $D^{(n)}$ and $\ddifn^+(D)$ natural $\dist$-actions.
As in \cite{KPX14}, for each $n \geq n(D)$, we can equip $D^{(n)}$, $D[1/t]^{(n)}$ with natural $\dist$-actions, which extends to $D$ and $D[1/t]$.
Also, for each $n \geq n(D)$, we can equip $\ddifn^+(D)$ with a natural $\dist$-action.
In fact, for any $n \geq 1$, we can equip with a natural $\dist$-action a finite generated $L_n[[t]]$-module $X$ with semilinear and continuous $\Gamma$-action with respect to the canonical Frech\'et topology as follows. 
Since one has $X=\varprojlim_{n}X/t^nX$ with the quotient $X/t^nX$ is a finite dimensional $L$-vector space with $L$-linear continuous $\Gamma$-action, it suffices to define a natural $\dist$-action on arbitrary finite dimensional $L$-vector space $M$ with an $L$-linear continuous $\Gamma$-action. 
First, it is easy to see that $\Gamma$-action on $M$ naturally extends to a continuous $\mathcal{O}_L[[\Gamma]]$-action.
Since $M$ is finite dimensional $L$-vector space, the the action of $\mathcal{O}_L[[\Gamma]][1/p]$-factors through a quotient $R_0$ of $\mathcal{O}_L[[\Gamma]][1/p]$ of finite length. 
Since the maximal ideals of $\mathcal{O}_L[[\Gamma]][1/p]$ bijectively correspond to closed maximal ideals of $\dist$, $R_0$ is also a quotient of $\dist$. i.e. the natural quotient map $\mathcal{O}_L[[\Gamma]][1/p]\rightarrow R_0$ factors through the inclusion $\mathcal{O}_L[[\Gamma]][1/p]\hookrightarrow \dist$. 

From now until the end of this section, we consider the following situation.
Let $D, D'$ be $\pg$-modules over $\calR_L$ such that $D[1/t] = D'[1/t]$. 
Then, we remark that one has $\ddifm(D)=\ddifm(D')$, $\ddr(D)=\ddr(D')$ and $\dcris(D)=\dcris(D')$.
Let  $\lambda \in \dist$ be any distribution.
We assume that, there exists some $n \geq \max \set{n(D),n(D')}$ such that we have
\[\lambda(\ddifm^+(D)) \subseteq \ddifm^+(D')\]
in $\ddifm(D)=\ddifm(D')$ for all $m \geq n$.

\begin{Prop}\label{morph}
For any $m \geq n$, we have $\lambda(D^{(m)}) \subseteq (D')^{(m)}$.
In particular, we have $\lambda(D) \subseteq D'$.
\end{Prop}

\begin{proof}
By \cite[II.1]{Ber08}, the submodule $(D')^{(m)} \subseteq (D')^{(m)}[1/t]=D^{(m)}[1/t]$ can be written as 
\[ (D')^{(m)} = \set{ x \in t^{-h}D^{(m)} : \iota_{m'}(x) \in \mathbf{D}^+_{\dif,m'}(D') \text{ for all $m' \geq m$}}, \]
where $h \in \Z_{>0}$ is a sufficient large integer.
Since $D^{(m)}$ is an $\dist$-module and $\iota_m$ commutes with $\dist$-action for any $m \geq n(D)$, one has 
$$\iota_{m'}(\lambda x)=\lambda \iota_{m'}(x)\in \mathbf{D}^+_{\dif,m'}(D')$$
for each $x\in D^{(m)}$ and $m'\geqq m$ by our assumption $\lambda(\ddifm^+(D)) \subseteq \ddifm^+(D')$ for all $m \geq n$, which shows 
that $\lambda x\in (D')^{(m)}$. 
\end{proof}

The following corollary is fundamental.
\begin{Coro}\label{cpxmorph}
Multiplying by $\lambda$ induces morphisms of complexes 
\[ C^\bullet_{\phi,\gamma}(D) \to C^\bullet_{\phi,\gamma}(D'), \quad C^\bullet_{\psi,\gamma}(D) \to C^\bullet_{\psi,\gamma}(D')\]
of $L$-vector spaces,
\[ C^\bullet_\psi(D) \to C^\bullet_\psi(D')\] 
of $\dist$-modules, and
\[ C^\bullet_\gamma(\ddifm^{(+)}(D)) \to C^\bullet_\gamma(\ddifm^{(+)}(D')) \]
of $L_m[[t]]$-modules for each $m \geq n$.
\end{Coro}

\begin{proof}
Since the operators $\phi, \, \psi$ are continuous so that they commute with the $\dist$-action, Proposition \ref{morph} gives our assertion.
\end{proof}

By abuse of notation, we use the same expression $\times \lambda$ for the morphisms defined in Proposition \ref{morph}, the ones in Corollary \ref{cpxmorph}, and the induced ones between their cohomologies, which will cause no confusion. We remark that the action $\times \lambda$ on 
$\mathrm{H}^i_{\gamma}(\ddif(D))=\mathrm{H}^i_{\gamma}(\ddif(D'))$, $\ddr(D)=\ddr(D')$ and $\dcris(D)=\dcris(D')$ is just the multiplication by $\lambda(\bold{1})\in L$. Here, 
for any $\lambda\in \dist$, we denote by $\lambda(\mathbf{1})\in L$ the image of 
$\lambda$ by the map $f_{\mathbf{1}} : \dist\rightarrow L : [\gamma]\mapsto 1$ ($\gamma\in \Gamma$). 

Recall the following morphisms defined in \cite{Nak14}:
\[ \can: \Hpg^1(D) \to \Hg^1(\ddif(D)) :\, [(x,y)] \mapsto [\iota_n(x)], \]
\[ g_D: \ddr(D) = \Hg^0(\ddif(D)) \to \Hg^1(\ddif(D)): \, \alpha \mapsto [\log \chi(\gamma) \alpha ] .\]
Since they commute with $\dist$-action, we immediately obtain the following lemma.
\begin{Lemm}\label{lambdacompati}
The action $\times \lambda$ induces the following commutative diagrams $:$ 
\[
\xymatrixcolsep{4pc}
\xymatrix{
\Hpg^1(D) \ar[r]^-{\times \lambda}  \ar@{->}[d]_-{\can}  & \Hpg^1(D') \ar@{->}[d]^-{\can} \\
\Hg^1(\ddif(D)) \ar[r]_-{\times \lambda(\bold{1})} & \Hg^1(\ddif(D')),
}
\xymatrixcolsep{4pc}
\xymatrix{
\ddr(D) \ar[r]^-{\times \lambda(\bold{1})}  \ar@{->}[d]_-{g_D}  & \ddr(D') \ar@{->}[d]^-{g_{D'}} \\
\Hg^1(\ddif(D))\ar[r]_-{\times \lambda(\bold{1})} & \Hg^1(\ddif(D')).
}
\]
\end{Lemm}

We next introduce a morphism of $L$-vector spaces
\[ \exp_D: t(D) \to \Hpg^1(D) \]
called the Bloch-Kato's exponential map, and if $D$ is de Rham, then we have another one
\[ \exp^*_D: \Hpg^1(D) \to \ddr^0(D)\]
called the Bloch-Kato's dual exponential map that is the Tate dual of $\exp_{D^*}$.
They are characterized by the following explicit formulae.

\begin{Theo}\label{ExplicitFormulae}
Let $D$ be a $\pg$-module over $\calR_L$.
\begin{enumerate}
	\item For $x \in \ddr(D)$, there exists $n \geq n(D)$ and $\tilde{x} \in D^{(n)}[1/t]^\Delta$ such that for any $m \geq n$ we have
	\[ \iota_m(\tilde{x}) - x \in \ddifm^+(D). \]
 	Using such an element $\tilde{x}$, we can calculate the value $\exp_D(x)$ as
 	\[  \exp_D(x) = [(\gamma-1)\tilde{x}, (\phi-1) \tilde{x}]   .\]

 	\item We assume that $D$ is de Rham.
 	Then $g_D$ is an isomorphism and $\exp^*_D$ is characterized by the following commutative diagram
	\[
	\xymatrixcolsep{4pc}
	\xymatrix{
	\ddr(D) \ar[r]^-{g_D}  \ar@{<-}[d]_-{\exp^*_D}  & \Hg^1(\ddif(D)) \ar@{<-}[d]^-{\can} \\
	\Hpg^1(D) \ar[r]_-{=} & \Hpg^1(D).
	}
	\]\end{enumerate}
\end{Theo}
\begin{proof}
See \cite[Section 2.3, 2.4]{Nak14} or \cite[Section 2B]{Nak17a} for the definition of $\exp_D, \exp^*_D$ and the proofs of the above formulae.
We note that our notation $\exp^*_D$ corresponds to $\exp^*_{D^*}$ in those papers.
\end{proof}

To state a relation between $\times \lambda$ and Bloch-Kato morphisms, we need some preparation.

For any $\gamma\in \Gamma \setminus \Gamma_{\mathrm{tor}}$, we set 
$$\omega\coloneqq\frac{1}{\mathrm{log}(\chi(\gamma))}\frac{d[\gamma]}{[\gamma]}=\frac{1}{\mathrm{log}(\chi(\gamma))}\frac{d([\gamma]-1)}{[\gamma]}\in \Omega^{1, \mathrm{an}}_{\dist/L}
\coloneqq\Gamma(\mathcal{W}, \Omega^1_{\mathcal{W}/L}).$$
This is independent of the choice of $\gamma$ since one has
$$\frac{1}{\mathrm{log}(\chi(\gamma^a))}\frac{d[\gamma]^a}{[\gamma]^a}=\frac{1}{a\mathrm{log}(\chi(\gamma))}a\frac{d[\gamma]}{[\gamma]}=\frac{1}{\mathrm{log}(\chi(\gamma))}\frac{d[\gamma]}{[\gamma]}$$ for any non zero $a\in \mathbb{Z}_p$. 
Then, one has
$\Omega^{1, \mathrm{an}}_{\dist/L}=\dist\omega$, which is a free $\dist$-module of rank one.  
For each $\lambda\in \dist$, we define 
$\frac{d\lambda}{\omega}\in \dist$ by $d\lambda=\frac{d\lambda}{\omega}\cdot \omega$. Explicitly, 
if $\gamma \in \Gamma_{\mathrm{free}}$ is a topological generator and $\lambda$ is of the form $\lambda=y\otimes f([\gamma]-1)$ 
with $y\in L[\Gamma_{\mathrm{tor}}]$ and $f(T)\in \mathcal{R}_L^+$, then one has 
$\frac{d\lambda}{\omega}=\mathrm{log}(\chi)(\gamma) y\otimes \frac{df}{dT}([\gamma]-1)$. 


In the following theorem, we shall compare the Bloch-Kato's morphisms of $D$ and $D'$ using $\lambda$.
It is the key lemma to prove our main theorem.

\begin{Theo}\label{keylemma}
\begin{enumerate}
	\item The diagram
	\[
	\xymatrixcolsep{4pc}
	\xymatrix{
	\ddr(D) \ar[r]^-{\times \lambda(\mathbf{1})}  \ar@{->}[d]_-{\exp_D}  & \ddr(D') \ar@{->}[d]^-{\exp_{D'}} \\
	\Hpg^1(D) \ar[r]_-{\times \lambda} & \Hpg^1(D')
	}
	\]
	commutes.
	\item Assume that $D$ or $D'$ is (thus both are) de Rham.
	Then the diagram
	\[
	\xymatrixcolsep{4pc}
	\xymatrix{
	\ddr(D) \ar[r]^-{\times \lambda(\mathbf{1})}  \ar@{<-}[d]_-{\exp^*_D}  & \ddr(D') \ar@{<-}[d]^-{\exp^*_{D'}} \\
	\Hpg^1(D) \ar[r]_-{\times \lambda} & \Hpg^1(D'),
	}
	\]
	commutes.
	\item Assume further that $\lambda(\mathbf{1})=0$
	Then the diagram	
	\[
	\xymatrixcolsep{4pc}
	\xymatrix{
	\ddr(D) \ar[r]^-{\times \frac{d\lambda}{\omega}(\mathbf{1})}  \ar@{<-}[d]_-{\exp^*_D}  & \ddr(D') \ar@{->}[d]^-{\exp_{D'}} \\
	\Hpg^1(D) \ar[r]_-{\times \lambda} & \Hpg^1(D')
	}
	\]
	commutes.
\end{enumerate}
\end{Theo}

\begin{proof}
First we prove (1).
Let $\alpha$ be an element of $\ddr(D)$.
By Theorem \ref{ExplicitFormulae}(1), there exist an integer $n \geq \max \{n(D),n(D')\}$ and an element $x \in D^{(n)}[1/t]^{\Delta}$ such that 
\[ \iota_m (x) - \alpha \in \mathbf{D}^+_{\dif,m}(D)\]
for any $m \geq n$.
Then one has
\[ \exp_D(\alpha) = [(\gamma -1)x,(\phi-1)x ] \in \Hpg^1(D) .\]
Thus, its image under the map $\times \lambda: \Hpg^1(D) \to \Hpg^1(D')$ is equal to
\[ [\lambda((\gamma -1)x), \lambda((\phi-1)x)] = [(\gamma -1)(\lambda x), (\phi-1)(\lambda x) ]\in \Hpg^1(D').\]
This is nothing but $\exp_{D'}(\lambda \alpha)$ because $\lambda x\in D'^{(n)}[1/t]^{\Delta}$ satisfies 
\[  \iota_m(\lambda x) - \lambda(\alpha) = \lambda(\iota_m (x) - \alpha) \in \mathbf{D}^+_{\dif,m}(D').\]
for any $m \geq n$ by Proposition \ref{morph}.

(2) follows immediately by Theorem \ref{ExplicitFormulae}(2) and Lemma \ref{lambdacompati}.

We shall prove (3).
Assume that $D$ and $D'$ are de Rham and $\lambda(\mathbf{1})=0$. 
We remark that the latter implies that one can write $\lambda p_{\Delta}=(\gamma-1)\lambda_0$ for some $\lambda_0\in \dist$. 
Let $[x,y] \in \mathrm{H}^1_{\phi,\gamma}(D)$ and put $\alpha = \exp^*_D\qty([x,y]) \in \ddr^0(D)$.
By replacing $n$ larger if necessary, we may assume that $x \in (D^{(n)})^\Delta$.
Take $m \geq n$ arbitrary.
By Theorem \ref{ExplicitFormulae}(2), one has 
$$[\iota_m(x)]=[\log \chi(\gamma) \alpha] \in \mathrm{H}^1_{\gamma}(\ddifm^+(D)),$$
and hence one obtains
$$\iota_m(x)-\log \chi(\gamma) \alpha\in (\gamma-1)\ddifm^+(D)^{\Delta}.$$
Applying $\lambda p_{\Delta} /(\gamma-1) =\lambda_0\in \dist$ on the both sides gives
$$\iota_m\qty(\frac{\lambda p_{\Delta}}{\gamma-1}(x)) - \mathrm{log}(\chi(\gamma))\frac{\lambda p_{\Delta}}{\gamma-1}\alpha\in \lambda p_\Delta(\ddifm^+(D)^{\Delta})\subseteq \ddifm^+(D')^{\Delta}. $$
Since one has
$$\mathrm{log}(\chi(\gamma))\frac{\lambda p_{\Delta}}{\gamma-1}\alpha=\mathrm{log}(\chi(\gamma))\frac{\lambda p_{\Delta}}{\gamma-1}(\mathbf{1})\alpha=
\frac{d\lambda}{\omega}(\mathbf{1})\alpha,$$ 
we obtain 
$$\iota_m\qty(\frac{\lambda p_{\Delta}}{\gamma-1}(x)) - \frac{d\lambda}{\omega}(\mathbf{1})\alpha\in \ddifm^+(D')^\Delta.$$

Since $\frac{\lambda p_{\Delta}}{\gamma-1}(x) \in (D^{(n)})^\Delta \subseteq (D^{(n)}[1/t])^\Delta = (D'^{(n)}[1/t])^\Delta$ and we have taken $m \geq n$ arbitrary, the explicit formula for $\exp_{D'}$ gives that
\begin{align*}
\exp_{D'}\qty(\frac{d\lambda}{\omega}(\mathbf{1}) \alpha) &= \qty[(\gamma -1) \frac{\lambda p_{\Delta}}{\gamma-1}(x), (\phi-1)\frac{\lambda p_{\Delta}}{\gamma-1}(x) ] \\
&=[\lambda x , \frac{\lambda p_{\Delta}}{\gamma-1}(\phi-1)(x)]\\
&=[\lambda x , \frac{\lambda p_{\Delta}}{\gamma-1}(\gamma-1)(y)]\\ 
&=\lambda [x, y],
\end{align*}
which proves (3).
\end{proof}

\section{Big exponential maps in the Local $\varepsilon$-conjecture for $\pg$-modules}
\label{3}

In this section, we first recall briefly the definition of the de Rham $\varepsilon$-isomorphisms for $\pg$-modules, and state the local $\varepsilon$-conjecture for cyclotomic deformations.
Then, we define the big exponential maps and study their several properties.

\subsection{de Rham $\varepsilon$-isomorphisms for $\pg$-modules}
\label{31}
We recall de Rham $\varepsilon$-isomorphisms over Robba rings following \cite{Nak17a}. 

First, for each local field $L / \Qp$ and each $\pg$-module $D$ over $\calR_L$, we define a graded line $\Delta_L(D)$ over $L$ called the fundamental line attached to $D$ as follows.

By \cite{Liu08}, the complex $C^\bullet_{\phi,\gamma}(D)$ is a perfect complex of $L$-vector spaces, and we put
\[ \Delta_{L,1}(D) = \Det_{L}(C^\bullet_{\phi,\gamma}(D)), \]
which is a graded line over $L$.
Here, $\Det_L$ is the determinant functor defined by Knudsen-Mumford \cite{KM76}.
We define another graded $L$-vector space $\Delta_{L,2}(D)$ as follows.
By the classification of rank $1$ $\pg$-modules over $\calR_L$ \cite[Proposition 3.1]{Col08}, there exists a unique continuous homomorphism $\delta_{\Det_{\calR_L}(D)}: \Qp^\times \to L^\times$ such that there exists an isomorphism $\Det_{\calR_L}(D) \cong \calR_L(\delta_{\Det_{\calR_L}(D)})$, and we define
\[ \mathcal{L}_L(D) = \set{x \in \Det_{\calR_L}(D) | \phi(x) = \delta_{\Det_{\calR_L}(D)}(p)x, \gamma(x) =  \delta_{\Det_{\calR_L}(D)}(\chi(\gamma))x\  (\gamma\in \Gamma)}, \]
which turns out to be an $L$-vector space of dimension $1$.
We then define an graded line over $L$
\[ \Delta_{D,2}(D) = (\mathcal{L}_L(D), r_D), \]
where we put $r_D = \rank_{\calR_L}(D)$.
Finally, we define a graded line $\Delta_L(D)$ over $L$ called its fundamental line by
\[ \Delta_L(D) = \Delta_{L,1}(D) \boxtimes \Delta_{L,2}(D).\]

We also define the fundamental line $\Delta^\Iw_L(D)$ for the cyclotomic deformation of a general $\pg$-module $D$ over $\calR_L$.
By \cite{KPX14}, the complex $C^\bullet_\psi(D)$ is perfect, thus we may define
\[ \Delta^{\Iw}_{L,1}(D) \coloneqq \Det_\dist \qty(C^\bullet_\psi(D)). \]
We also define 
\[\Delta^{\Iw}_{L,2}(D)\coloneqq \Delta_{L,2}(D) \otimes_L \dist, \]
and define the fundamental life for the cyclotomic deformation
\[ \Delta^{\Iw}_L(D) \coloneqq \Delta^{\Iw}_{L,1}(M) \boxtimes \Delta^{\Iw}_{L,2}(M).\]
Recall that for any continuous character $\delta: \Gamma \to L^\times$, we can consider a $\pg$-module $D(\delta) = D e_\delta$ with a formal element $e_\delta$ on which we have
\[ \phi(x e_\delta) = \phi(x) e_\delta, \, \gamma (x e_\delta) = \delta(\gamma)\gamma(x) e_\delta. \]
In particular, we put $D^* = \Hom_{\calR_L}(D,\calR_L)(\chi)$.

As studied in Subsection 4A of \cite{Nak17a}, one has canonical isomorphisms 
\[ \ev_{\delta,j}: \Delta^{\Iw}_{L,j}(D) \otimes_{f_\delta} L \xrightarrow{\sim} \Delta_{L,j}(D(\delta)), \,\can_{\delta,j}: \Delta^\Iw_{L,j} (D) \otimes_{g_\delta} \dist \xrightarrow{\sim} \Delta^\Iw_{L,j} (D(\delta)) \]
for $j = 1, 2, \emptyset$, where $f_\delta: \dist \to L$ (resp. $g_\delta: \dist \to \dist$) is the continuous homomorphism of $L$-algebra extending $: \gamma \mapsto \delta^{-1}(\gamma)$ (resp. $\gamma \mapsto \delta^{-1}(\gamma)\gamma$).

The local $\varepsilon$-conjecture concerns canonical bases of the fundamental lines for de Rham $\pg$-modules, which we recall briefly as follows: see \cite{Nak17a} for the precise definition.
Let $M$ be a de Rham $\pg$-module over $\calR_L$.
Set $1_L = (L,0)$ as the trivial line.
We define the following two isomorphisms
\begin{align*}
\theta_{\dR}(M)&:1_L \xrightarrow{\sim} \Delta_{L,1}(M) \boxtimes \Det_{L}(\ddr(M)), \\
f_M&: \Delta_{L,2}(M) \xrightarrow{\sim} \Det_{L}(\ddr(M)).
\end{align*}

To define the isomorphism $\theta_{\dR}(M)$, we first recall that there exist exact sequences of $L$-vector spaces
\[ C^\bullet_1(M) : 0 \to \Hpg^0(M) \to \dcris(M) \xrightarrow{x \mapsto ((1-\phi)x,\bar{x})} \dcris(M) \oplus t(M) \to \Hpg^1(M)_f \to 0,\]
\[ C^\bullet_2(M) : 0 \to \Hpg^1(M)_{/f} \to \dcris(M^*)^{\lor} \oplus \ddr^0(M) \to \dcris(M^*)^\lor \to \Hpg^2(M) \to 0\]
obtained by the Bloch-Kato's fundamental sequences and Tate duality, both of which are concentrated in $[0,4]$.
We define the canonical isomorphism $\theta_{\dR}(M)$ as the inverse of the isomorphism
\begin{multline*} \theta_{\dR}(M) : \Delta_{L,1}(M) \boxtimes \Det_L(\ddr(M)) \\ 
\xrightarrow{(\sharp)} \Det_L(C^\bullet_1(M))^{[-1]} \boxtimes \Det_L(C^\bullet_2(M)) \xrightarrow{(\flat)} 1_L \boxtimes 1_L \xrightarrow{\can} 1_L,
\end{multline*}
where the isomorphism $(\sharp)$ is defined by cancellation $X \boxtimes X^{-1} \isom 1_L: a \otimes f \mapsto f(a)$ for each graded invertible line $X$ and the one $(\flat)$ is by the trivializations via the determinant functor.

Next we define the isomorphism $f_M$.
Since $M$ is de Rham, we have $\ddif^+(M) = \ddr(M) \otimes_{L} L_{\infty}[[t]]$ where $L_\infty[[t]] = \cup_{n \geq 1} L_n[[t]]$.
By Lemma.3.4 of \cite{Nak17a}, a map
\[ \mathcal{L}_L(M) \to \ddifn(\det_{\calR_L}(M)): x \mapsto \frac{1}{\varepsilon_L(W(M))} \frac{1}{t^{h_M}}x  \]
for sufficient large $n$ induces an isomorphism $f_M:\Delta_{L,2}(M) \xrightarrow{\sim} \Det_L(\ddr(M))$.
Here, the constant $\varepsilon_L(W(M)) \in L_\infty$ is defined by using the Weil-Deligne representation $W(M)$ attached to $M$ and the fixed basis $(\zeta_{p^n})_n \in \Zp(1)$, via the theory of $\varepsilon$-constants of Deligne-Langlands \cite{Del73}, Fontaine-Perrin-Riou \cite{FP94}.  

Using $\theta_\dR(M)$ and $f_M$, we define 
\[ \varepsilon^\dR_{L}(M) = (\id \boxtimes f_M^{-1}) \circ (\Gamma(M)\theta_\dR(M)): 1_L \xrightarrow{\sim} \Delta_L(M) \]
and call it the de Rham $\varepsilon$-isomorphism for $M$.
Here, the $\Gamma$-constant $\Gamma(M)$ for $M$ is defined by $\Gamma(M) = \prod_{1 \leq i \leq r} \Gamma^*(h_i)^{-1}$, where for $r \in \Z$ we put
\[ \Gamma^*(r) = 
\begin{cases}
(r-1)! & (r \geq 1) \\
\dfrac{(-1)^{r}}{(-r)!}& (r \leq 0).
\end{cases}
\]

Now we can state the local $\varepsilon$-conjecture for cyclotomic deformation for $\pg$-modules.
\begin{Conj}\label{Conjecture for Dfm}
For each finite extension $L/ \Qp$ and each de Rham $\pg$-module $D$ over $\calR_L$, there exists an isomorphism
\[\varepsilon^\Iw_L(D): 1_\dist \xrightarrow{\sim} \Delta^\Iw_L(D)\]
satisfying the following  commutative diagram
\[
\xymatrixcolsep{4pc}
\xymatrix{
\Delta^\Iw_L(D) \otimes_{f_\delta} L\ar[r]^-{\ev_\delta}  \ar[d]_-{\varepsilon^\Iw_L(D) \otimes \id}  & \Delta_L(D(\delta)) \ar@{->}[d]^-{\varepsilon^{\mathrm{dR}}_L(D(\delta))} \\ 
1_\dist \otimes_{f_\delta} L \ar[r]_-{\can} & 1_L.
}
\]
for arbitrary de Rham continuous characters $\delta: \Gamma \to L^\times$.
\end{Conj}
Since the set of all the de Rham characters is Zariski dense in the weight space $\mathcal{W}$, the isomorphism $\varepsilon^\Iw_L(D)$ is uniquely determined (if it exists).

\subsection{Big exponential maps}
\label{32}

Throughout this section, let $L$ be a finite extension of $\Qp$, $M$ a de-Rham $(\phi, \Gamma)$-modules of rank $r$ over $\calR^+_L$, and $N = \nrig(M)$ its associated $p$-adic differential equation define by Berger in \cite{Ber08}.
Note that $N$ is characterized as the $\pg$-module in $M[1/t]$ satisfying $N[1/t] = M[1/t]$ and 
\[ \ddifn^+(N) = L_n[[t]] \otimes_L \ddr(M) \]
for a sufficient large $n$.

In this subsection, we construct the big exponential map of $M$
\[\Exp(M): \Delta^{\Iw}_L(N) \xrightarrow{\sim} \Delta^{\Iw}_L(M),\]
and prove its properties.
Its construction involves the theory of big exponential map, especially the $\delta(D)$-theorem studied in \cite{Nak14}, which generalizes the original $\delta(V)$-conjecture in \cite{Per94}.

First, we shall construct 
\[\Exp_1:\Delta^\Iw_{L,1}(N) \xrightarrow{\sim} \Delta^\Iw_{L,1}(M)\]
as follows.

Since all the complexes $\Hpsi^1(D)_\tors[0]$, $\Hpsi^1(D)[0]$ and $\Hpsi^2(D)[0]$ are perfect, there exist canonical isomorphisms 
\begin{align*}
&\Det_{\dist}(C^\bullet_\psi(D)) \cong \Det_{\dist}(\Hpsi^1(D)[0])^{-1} \boxtimes \Det_{\dist}(\Hpsi^2(D)[0])\\
&\cong \Det_{\dist}(\Hpsi^1(D)_{\free}[0])^{-1} \boxtimes \Det_{\dist}(\Hpsi^1(D)_\tors[0])^{-1} \boxtimes \Det_{\dist}(\Hpsi^2(D)[0]),
\end{align*}
where we set $\Hpsi^1(D)_{\free} = \Hpsi^1(D)/\Hpsi^1(D)_\tors$.
Extending the coefficients to the total fraction ring $Q(\dist)$ of $\dist$, we have
\begin{align*}
&\Det_{\dist}(C^\bullet_\psi(D)) \otimes_{\dist} Q(\dist) \\
&\cong \Det_{Q(\dist)} \qty(\Hpsi^1(D)_{\free} \otimes Q(\dist)[0])^{-1} \boxtimes (1_{Q(\dist)})^{-1} \boxtimes 1_{Q(\dist)} \\
&= \qty( \bigwedge^{r} \qty( \Hpsi^1(D)_{\free} \otimes_\dist  Q(\dist))^{-1}, \, -r),
\end{align*}
under which the image of $\Det_{\dist}(C^\bullet_\psi(D))$ is calculated as
\[ \bigwedge^{r} \Hpsi^1(D)_{\free} ^{-1} \cdot \chara_\dist (\Hpsi^1(D)_\tors) \cdot \chara_\dist \qty(\Hpsi^2(D))^{-1} .\]

On the other hand, let $h>0$ be a sufficient large integer satisfying $\ddr^{-h}(M) =\ddr(M)$.
Then by \cite[Lemma 3.6]{Nak14} we have $\left(\prod^{h-1}_{i=0}\nabla_{i}\right)(\ddifm^+(N)) \subseteq \ddifm^+(M)$ for any $m \geq n(M)$, and we define a morphism $\Exp_{(h)}(M): \Hpsi^1(N) \to \Hpsi^1(M)$ as the induced one by Corollary \ref{cpxmorph};
\[ \Exp_{(h)}(M)\coloneqq \times \prod^{h-1}_{i=0}\nabla_{i}:  \Hpsi^1(N) \to \Hpsi^1(M). \]
It induces an injective map $\overline{\Exp_{(h)}}(M): \Hpsi^1(N)_\free \to \Hpsi^1(M)_\free$, which turns out to be injective.
Since $\nabla_i \in \dist$ is a non-zero-divisor for any integer $i$, we can define a modified map $\overline{\Exp}(M):\wedge^{r}\Hpsi^1(N)_\free \otimes_\dist Q(\dist) \to \wedge^{r} \Hpsi^1(M)_\free \otimes_\dist Q(\dist)$ by 
\[\overline{\Exp}(M) = \bigwedge^r \overline{\Exp_{(h)}}(M) \otimes \frac{1}{\displaystyle\prod^r_{i=1} \prod^{h-h_i-1}_{j_i=1}\nabla_{h_i+j_i}} \cdot \id_{Q(\dist)} ,\]
where $h_1, \dots, h_r$ are the Hodge-Tate weights of $M$ with multiplicity.
Note that the right hand side doesn't depend on $h$, which justifies our notation $\overline{\Exp}(M)$.

To define $\Exp_1(M)$, the main part of $\Exp(M)$, the following theorem is essential.
It is nothing but theorem $\delta(D)$ in the context of the local $\varepsilon$-conjecture.
\begin{Theo}\label{delta(D)}
$\overline{\Exp}(M)$ is an isomorphism of $Q(\dist)$-modules.
Moreover, by restriction, it induces an isomorphism of $\dist$-modules
\begin{align*}
&\bigwedge^r \Hpsi^1(N)_\free \otimes \chara_\dist \qty(\Hpsi^1(N)_\tors)^{-1} \cdot \chara_\dist \qty(\Hpsi^2(N)) \\ 
&\longrightarrow \bigwedge^r \Hpsi^1(M)_\free \otimes \chara_\dist (\Hpsi^1(M)_\tors)^{-1} \cdot \chara_\dist \qty(\Hpsi^2(M)).
\end{align*}
\end{Theo}
\begin{proof}
Since $\overline{\Exp}(M)$ is a multiplication of a product of non-zero divisors $\nabla^\pm_i \in Q(\dist)$ with $i \in \Z$, it is an isomorphism as $Q(\dist)$-modules.

For the latter assertion, we first remark that in $\wedge^{r} \Hpsi^1(M)_\free \otimes_\dist Q(\dist)$ we have
\begin{multline*}
\overline{\Exp}(M)\qty( \bigwedge^r \Hpsi^1(N)_\free) =  \overline{\Exp_{(h)}}(M) \qty(\bigwedge^r \Hpsi^1(N)_\free) \otimes \qty( \displaystyle\prod^r_{i=1} \prod^{h-h_i-1}_{j_i=1}\nabla_{h_i+j_i})^{-1} \\
=   \det_{\dist} (\overline{\Exp_{(h)}}(M)) \cdot \bigwedge^r \Hpsi^1(M)_\free \otimes \qty( \displaystyle\prod^r_{i=1} \prod^{h-h_i-1}_{j_i=1}\nabla_{h_i+j_i})^{-1},
\end{multline*}
where $\det_{\dist} (\overline{\Exp_{(h)}}(M)) \subseteq \dist$ is the determinant ideal of $\overline{\Exp_{(h)}}(M)$.
Therefore, the claim is equivalent to the equality 
\begin{multline*}
\qty( \displaystyle\prod^r_{i=1} \prod^{h-h_i-1}_{j_i=1}\nabla_{h_i+j_i})^{-1} \det_{\dist} (\overline{\Exp_{(h)}}(M)) \cdot \chara_\dist \qty(\Hpsi^1(N)_\tors)^{-1} \cdot \chara_\dist \qty(\Hpsi^1(M)_\tors)     \\
= \chara_\dist \qty(\Hpsi^2(M)) \cdot \chara_\dist \qty(\Hpsi^2(N))^{-1}
\end{multline*}
of fractional ideals in $Q(\dist)$, which is proved as the theorem $\delta(D)$ \cite[Theorem 3.14.]{Nak14}.
\end{proof}

\begin{Defi}
We define an isomorphism 
\[ \Exp_1(M): \Delta^\Iw_{L,1}(N) \xrightarrow{\sim} \Delta^\Iw_{L,1}(M) \]
as the isomorphism corresponding to the one appearing Theorem \ref{delta(D)} under the functor $[-1]$.
\end{Defi}

Second, we shall define 
\[ \Exp_2(M):\Delta^\Iw_{L,2}(N) \xrightarrow{\sim} \Delta^\Iw_{L,2}(M).\]

\begin{Lemm}\label{lem_keyExp2}
Under the canonical identification of $\det_{\calR_L}(M)[1/t]$ and $\det_{\calR_L}(N)[1/t]$, we have
\[ \det_{\calR_L} (N) = \nrig(\det_{\calR_L} (M)) = t^{-h_M} \det_{\calR_L}(M).  \]
\end{Lemm}
\begin{proof}
The first equality follows from $\det_{L}(\ddr(M))= \ddr(\det_{\calR_L}(M))$.
The second one follows from the fact that for a general $1$-dimensional $\pg$-module $D$ corresponding a continuous character $\delta$, we have $ \nrig(D) = t^{-h_D} D$.
This shows the second equality.
\end{proof}

Lemma \ref{lem_keyExp2} justifies the following definition.
\begin{Defi}\label{def_Exp2}
We define the isomorphism
\[\Exp_2(M):\Delta^\Iw_{L,2}(N) \xrightarrow{\sim} \Delta^\Iw_{L,2}(M)\]
as the scalar extension of the isomorphism 
\[ \mathcal{L}_2(N) \xrightarrow{\sim} \mathcal{L}_2(M); \, x \mapsto (-t)^{h_M} x.\]
\end{Defi}

We define the big exponential map of $M$ as follows.
\begin{Defi}\label{def_Exp}
We define the isomorphism 
\[ \Exp(M): \Delta^\Iw_L(N) \xrightarrow{\sim} \Delta^\Iw_L(M) \]
as the product $\Exp(M)\coloneqq \Exp_1(M)\boxtimes\Exp_2(M)$ and call it the big exponential map of $M$.
\end{Defi}

We also define relative big exponential maps, which are useful to prove our main theorem.
\begin{Defi}
Let $M'$ be another $\pg$-module such that $M[1/t] = M'[1/t]$.
We define the isomorphism
\[ \Exp_j(M,M'): \Delta^\Iw_{L,j}(M) \xrightarrow{\sim} \Delta^\Iw_{L,j}(M') \]
as the composition $\Exp_j(M') \circ \Exp_j(M)^{-1}$ for each $j=1,2,\emptyset$.
\end{Defi} 
We note that the definition of $\Exp_j(M,M')$ is justified by the equality $\nrig(M) = \nrig(M')$.

The following proposition is used when we reduce the proof of our main theorem to the generic case.

\begin{Prop}\label{exactness of Exp}
Let $ C^\bullet : 0 \to M_1 \to M_2 \to M_3 \to 0$ be an exact sequence of de Rham $\pg$-modules, and $\nrig(C^\bullet): 0 \to N_1 \to N_2 \to N_3 \to 0$ the exact one corresponding $C^\bullet$ via the functor $\nrig$.
Then we have
\[
\xymatrixcolsep{4pc}
\xymatrix{
\Delta^\Iw_{L,j}(M_1) \boxtimes \Delta^\Iw_{L,j}(M_3) \ar[r]^-{}  \ar@{->}[d]_-{\Exp_j(M_1) \boxtimes \Exp_j(M_3)}  &\Delta^\Iw_{L,j}(M_2) \ar@{->}[d]^-{\Exp_j(M_2)} \\
\Delta^\Iw_{L,j}(N_1) \boxtimes \Delta^\Iw_{L,j}(N_3 ) \ar[r]_-{} & \Delta^\Iw_{L,j}(N_2),
}
\]
for $j=1,2,\emptyset$, where the horizontal isomorphisms are induced by $C^\bullet$ and $\nrig(C^\bullet)$ respectively.
\end{Prop}
\begin{proof}
The Hodge-Tate weights of $M_2$ is the same as the union of the ones of $M_1$ and $M_3$ with multiplicity, and thus we have $h_{M_2} = h_{M_1} + h_{M_3}$.
This gives the commutativity for each $j=1,2$ by the definition of $\Exp_j$ and so for $j = \emptyset$.
\end{proof}

Big exponential maps are compatible with twists by characters on $\Gamma$ as follows.

\begin{Lemm}\label{Exp and twist universally}
Let $\delta: \Gamma \to L^\times$ be a de Rham character.
Then the diagram
\[
\xymatrixcolsep{4pc}
\xymatrix{
\Delta^\Iw_{L,j} (M) \otimes_{g_\delta} \dist \ar[r]^-{\can_{\delta,j}}  \ar@{->}[d]_-{\Exp_j(M,M') \otimes \id}  &\Delta^\Iw_{L,j} (M(\delta))  \ar@{->}[d]^-{\Exp_j(M(\delta),M'(\delta))} \\
 \Delta^\Iw_{L,j}(M') \otimes_{g_\delta}\dist  \ar[r]_-{\can_{\delta,j}} & \Delta^\Iw_{L,j}(M'(\delta)),
}
\]
commutes for $j = 1,2,\emptyset$.
\end{Lemm}
\begin{proof}
The case $j=2$ can be checked easily by definition.
The case $j=1$ follows from the facts that the first term of a big exponential map are induced by the multiplication of a product of $\nabla^{\pm}_i$ for $i \in \Z$ by definition, and that
\[ g_\delta(\nabla_{i}) = g_\delta \qty( \frac{\log \gamma}{\log \chi(\gamma)} - i) = \frac{\log \delta^{-1}(\gamma)\gamma}{\log \chi(\gamma)} - i = \qty( \frac{\log \gamma}{\log \chi(\gamma)} - k ) - i = \nabla_{i + k}, \]
where $k$ is the Hodge-Tate weight of $\delta$.
\end{proof}

\section{Interpolation formula of $\Exp(M)$ for local $\varepsilon$-isomorphisms}
\label{4}

In this section, we first state the main result and its corollary.
Its proof will be divided into the next three subsections.
We utilize an explicit construction of $\varepsilon^\Iw_L$-isomorphisms for rank $1$ $\pg$-modules, which is one of the main results in \cite{Nak17a}.

\subsection{Statement of main result}
\label{41}

Let $L$ be a finite extension of $\Qp$, $M$ a de-Rham $(\phi,\Gamma)$-module over $\calR_L$, and $N = \nrig(M)$ the $p$-adic differential equation corresponding to $M$.
For any character $\delta: \Gamma \to L^\times$, we denote $\Exp_j(M)_{\delta}$ as the isomorphism commuting the diagram
\[
\xymatrix{
\Delta_{L, j}(N(\delta)) \ar[rr]^-{ \Exp_j(M)_\delta } &&  \Delta_{L, j}(M(\delta))\\
\ar[u]^-{\ev_\delta}  \Delta^\Iw_{L,j}(N) \otimes_{f_\delta} L \ar[rr]_-{\Exp_j(M) \otimes \id } && \Delta^\Iw_{L,j}(M) \otimes_{f_\delta} L \ar[u]_-{\ev_\delta} \\
}
\]
for each $j= 1,2,\emptyset$.

The following is the main theorem of this paper.
\begin{Theo}\label{main}
For any de Rham character $\delta: \Gamma \to L^\times$, the diagram
\[\xymatrix{
& 1_L \ar[dl]_-{\varepsilon^{\dR}_L(N(\delta))}\ar[dr]^-{\varepsilon^{\dR}_L(M(\delta))}\ar@{}[d]|{} & \\
\Delta_{L}(N(\delta)) \ar[rr]_-{ \Exp(M)_\delta}   && \Delta_L(M(\delta))\\
}
\]
commutes.
\end{Theo}
Since de Rham $\varepsilon$-isomorphisms are compatible with base change, a similar statement for any de Rham character $\delta: \Gamma \to \overline{\Q}_p^\times$ is deduced from the above case by enlarging $L$ if necessary.

\begin{Rema}\label{interpolates all the twists}
Since $\varepsilon^\dR$-isomorphisms consist particularly of Bloch-Kato's exponential maps and dual exponential maps, Theorem \ref{main} can be regarded as a generalized interpolation formula of big exponential maps in the context of the local $\varepsilon$-conjecture; our theorem treats general de Rham $\pg$-modules and covers all of the twists by de Rham characters on $\Gamma$, that is, $\chi^k \tilde{\delta}$ for any $k \in \Z$ and any finite character $\tilde{\delta}$.

We also remark that, in a case such as $\dcris(M^*)^{\phi = 1} \neq 0$, our theorem gives a non-trivial information of another exponential map $\exp_{f,M}: \dcris(M) \to \Hpg^1(M)$, by which we can study exceptional zeros of $p$-adic $L$-functions (See \cite{Ben14} for example.)
\end{Rema}

The following corollary is an important consequence.
\begin{Coro}\label{main'}
The existence of $\varepsilon^{\Iw}_L(M)$ is equivalent to that of $\varepsilon^{\Iw}_L(N)$ for $N=\nrig(M)$. 
More precisely, 
if one of them exists, then the other one also exists and we have the following commutative diagram:
\[\xymatrix{
& 1_{\dist}\ar[dl]_-{\varepsilon^{\Iw}_L(N)}\ar[dr]^-{\varepsilon^{\Iw}_L(M)}\ar@{}[d]|{}&\\
\Delta^\Iw_L(N)  \ar[rr]_-{ \Exp(M) }&& \Delta^\Iw_L(M).
}
\]
\end{Coro}
\begin{proof}
If $\varepsilon^{\Iw}_L(N)$ (resp. $\varepsilon^{\Iw}_L(M)$) exists, then we define $\varepsilon^{\Iw}_L(M)$ (resp. $\varepsilon^{\Iw}_L(N)$) by 
$$\varepsilon^{\Iw}_L(M) : =\Exp(M) \circ \varepsilon^{\Iw}_L(N) \ \ (\text{resp}. \varepsilon^{\Iw}_L(N) : =\Exp(M)^{-1} \circ \varepsilon^{\Iw}_L(M)).$$
Since the isomorphism $\varepsilon^{\Iw}_L(N)$ (resp. $\varepsilon^{\Iw}_L(M)$) satisfies the commutative diagram in Conjecture \ref{Conjecture for Dfm} for arbitrary 
de Rham character $\delta$ by assumption, the isomorphism $\varepsilon^{\Iw}_L(M)$ (resp. $\varepsilon^{\Iw}_L(N)$) also satisfies 
the commutative diagram for arbitrary de Rham $\delta$ (in Conjecture) by Theorem \ref{main}, which shows that 
$\varepsilon^{\Iw}_L(M)$ (resp. $\varepsilon^{\Iw}_L(N)$) satisfies the conjecture. 
\end{proof}

By this corollary, the conjecture for all the de Rham $(\phi,\Gamma)$-modules is reduced to that for de Rham $(\phi,\Gamma)$-modules with a structure of $p$-adic differential equation (equivalently, de Rham $(\phi,\Gamma)$-modules with all Hodge-Tate weights $0$). 
This equivalence was in fact effectively used to prove the conjecture for rank $1$ case in \cite{Nak17a} (see also Remark \ref{rank 1 and conjecture}).

\begin{Rema}\label{main for cris}
Assume that $M$ is crystalline.
We remark that Theorem \ref{main'} gives an alternative construction of $\varepsilon^\Iw_L(M)$ (cf. \cite{Nak17a}).
In this case, the canonical map
\[ \calR_L \otimes_L \dcris(M) \to N \]
is an isomorphism as $\pg$-modules, and we can easily construct $\varepsilon^\Iw_L(N)$; its scalar extension with respect to the canonical homomorphism $\dist \to \calR_L(\Gamma)$ is induced by a composition of isomorphisms
\begin{multline*}
\calR_L(\Gamma) \otimes_L \dcris(M) \cong \calR_L^{\psi=0} \otimes_L \dcris(M)  \\ \cong \qty(\calR_L \otimes_L \dcris(M))^{\psi=0} \cong N^{\psi=0} \cong N^{\psi=1} \otimes_\dist \calR_L(\Gamma),
\end{multline*}
where the first isomorphism is obtained by the one
\[ \calR_L(\Gamma) \cong (\calR_L)^{\psi=0}; \lambda \mapsto \lambda \qty( (1+T)^{-1}), \]
and the last isomorphism is obtained by the map
\[ 1-\phi: N^{\psi=1} \to N^{\psi=0}. \]
Thus, using Theorem \ref{main'}, we obtain another construction of $\varepsilon^\Iw_L(M)$. 
\end{Rema}

Before proving the main theorem, we shall state an equivalent version of Theorem \ref{main}.
As before, let $L$ be a finite extension of $\Qp$ and $M,M'$ de-Rham $(\phi,\Gamma)$-modules over $\calR_L$ with $M[1/t] = M'[1/t]$.
For any character $\delta: \Gamma \to L^\times$, we denote $\Exp_j(M,M')_{\delta}$ as the isomorphism commuting the diagram
\[
\xymatrix{
\Delta_{L, j}(M(\delta)) \ar[rr]^-{ \Exp_j(M,M')_\delta } &&  \Delta_{L, j}(M'(\delta))\\
\ar[u]^-{\ev_\delta}  \Delta^\Iw_{L,j}(M) \otimes_{f_\delta} L \ar[rr]_-{\Exp_j(M,M') \otimes \id} && \Delta^\Iw_{L,j}(M') \otimes_{f_\delta} \ar[u]_-{\ev_\delta} \\
}
\]
for each $j= 1,2,\emptyset$.
\begin{Theo}\label{main relative}
For any de Rham character $\delta: \Gamma \to L^\times$, the diagram
\[\xymatrix{
& 1_L \ar[dl]_-{\varepsilon^{\dR}_L(M(\delta))}\ar[dr]^-{\varepsilon^{\dR}_L(M'(\delta))}\ar@{}[d]|{} & \\
\Delta_{L}(M(\delta)) \ar[rr]_-{ \Exp(M,M')_\delta}   && \Delta_L(M'(\delta))\\
}
\]
commutes.
\end{Theo}

We shall prove Theorem \ref{main} in the rest of the paper as follows.
In the subsection \ref{42}, we prove Theorem \ref{main} for rank $1$ $\pg$-modules.
In the subsection \ref{43}, we introduce a special class of $\pg$-modules called generic, and reduce the proof of \ref{main} for general $\pg$-modules to that for generic ones using the result for rank $1$ case.
In the final subsection \ref{44}, we complete the proof of Theorem \ref{main} by proving Theorem \ref{main relative} for generic $\pg$-modules.

\subsection{Proof for rank one case}\label{42}

We prove Theorem \ref{main} when $M$ is of rank $1$.
We utilize the explicit construction of $\varepsilon^\Iw_L(M)$ obtained in \cite{Nak17a}.

\begin{Theo}\label{main rank 1}
When $M$ is of rank $1$, the diagram of Theorem \ref{main} commutes.
\end{Theo}
\begin{proof}
By Theorem 3.11 of \cite{Nak17a}, the isomorphisms $\varepsilon^\Iw_{L}(M)$ and $\varepsilon^\Iw_L(N)$ exist.
Moreover, since we have $N = t^{-h_M}M$, it suffices to show that the diagram
\[\xymatrix{
& 1_{\dist}\ar[dl]_-{\varepsilon^{\Iw}_L(M)}\ar[dr]^-{\varepsilon^{\Iw}_L(tM)}\ar@{}[d]|{}&\\
\Delta^\Iw_L(M)  \ar[rr]_-{ \Exp(M,tM) }&& \Delta^\Iw_L(tM)
}
\]
commutes.

By the explicit construction in Section 4A of \cite{Nak17a}, for a general $\pg$-module $D$ of rank $1$, the isomorphism
\[ \varepsilon^\Iw_L(D) \otimes_{\dist} \id_{\calR_L(\Gamma)}: 1_{\calR_L(\Gamma)} \isom \Delta^\Iw_L(D) \otimes_{\dist}  {\calR_L(\Gamma)} \]
are obtained by the isomorphisms
\[ \theta_1 = 1-\phi: \Delta^\Iw_{L,1}(M) \otimes_\dist \calR_L(\Gamma) \cong \qty((\calR_L e_\delta)^{\psi=0},1)^{-1},\]
\[ \theta_2: \calR_L(\Gamma) \otimes_L L e_{\delta_D} \isom (\calR_L e_\delta)^{\psi=0}; \lambda \otimes e_{\delta_D} \mapsto \lambda( (1+X)^{-1} e_{\delta_D}),\]
where we put $\delta_D: \Qp^\times \to L^\times$ as the character corresponding to $D$.
Since $\Exp_1(M,tM)$ is induced by multiplying $\nabla_{h_M}$, and we can calculate
\begin{align*}
\nabla_{h_M}(\lambda( (1+X)^{-1} e_{\delta_M})) &= \lambda ((\nabla_0( (1+X)^{-1} e_{\delta_M})) -h_M (1+X)^{-1} e_{\delta_M}) \\
&= \lambda( -t(1+X)^{-1} e_{\delta_M} + (1+X)^{-1} (h_M e_{\delta_M}) -h_M (1+X)^{-1} e_{\delta_M}) \\
&= - \lambda ( (1+X)^{-1} te_{\delta_M}),
\end{align*}
our assertion follows from the equality $\Exp_2(M,tM)(e_{\delta_M})= - te_{\delta_M}$.
\end{proof}

\begin{Rema}\label{rank 1 and conjecture}
Theorem \ref{main rank 1} shows that our main theorem exactly generalizes the Proposition 4.13 in \cite{Nak17a}, which is proved in a different way and used in the proof of the local $\varepsilon$-conjecture for rank $1$ $\pg$-modules.
\end{Rema}

\subsection{Reduction to generic case}
\label{43}
In this subsection, we define genericity of a $(\phi,\Gamma)$-module and reduce the proof of our main theorem for the general case to that the generic case. 

\begin{Defi}\label{defigen}
A $(\phi, \Gamma)$-module $D$ over $\calR_L$ is generic if for any character $\delta: \Gamma \to \overline{\Q}_p^\times$ we have $\dcris(D(\delta))=0$ and $\dcris(D(\delta)^*)=0.$
\end{Defi}

For general $\pg$-modules $D, D'$ with $D[1/t] = D'[1/t]$, $D$ is generic if and only if $D'$ is generic by definition of $\dcris$.
In particular, if a generic $(\phi, \Gamma)$-module $M$ is de Rham, then the attached $p$-adic differential equation $N=\nrig(M)$ is also generic.

\begin{Lemm}\label{lemma:vanishing H2 when generic}
Let $D$ be a generic $(\phi, \Gamma)$-module over $\calR_L$ of rank $r$. 
Then we have $\Hpsi^2(D)=\Hpsi^2(D^*)=0$.
\end{Lemm}
\begin{proof}
We shall show $\Hpsi^2(D) = 0$.
For any continuous character $\delta: \Gamma \to \overline{\Q}_p^\times$, we have $\Hpg^0(D(\delta)^*) \subseteq (D(\delta)^*)^\Gamma \subseteq \dcris(D(\delta)^*) = 0$, and the Tate duality gives $\Hpg^2(D(\delta))=0$. 
On the other hand, we have $\Hpsi^2(D) \otimes_\dist (\dist/\mathfrak{m}_\delta) \cong D^\Delta/(\psi-1,\mathfrak{m}_\delta) \cong \Hpg^2(D(\delta^{-1})) = 0$ where $\mathfrak{m}_\delta \subseteq \dist$ is the corresponding maximal ideal. 
Since the set $\{\mathfrak{m}_{\delta}\}_{\delta : \Gamma \to \overline{\Q}_p^\times}$ coincides with the set of all the maximal ideals of $\dist$ and $\Hpsi^2(D)$ is a torsion coadmissible $\dist$-module, it gives $\Hpsi^2(D)=0$.
\end{proof}

The cohomologies of a $\pg$-module whose second $\psi$-cohomology and that of its dual vanish are quite simple.
\begin{Lemm}\label{lemma: generic treats easily}
Let $D$ be a $(\phi, \Gamma)$-module over $\calR_L$ of rank $r$ such that $\Hpsi^2(D)=\Hpsi^2(D^*)=0$.
Then one has $\Hpg^i(D)=0$ for $i=0,2$, and $\dim_L(\Hpg^1(D)) = r$.
Also, the first $\psi$-cohomology $\Hpsi^1(D)$ is free of rank $r$ over $R_L^+(\Gamma)$ and for any continuous character $\delta: \Gamma \to L^\times$, one has a canonical isomorphism $\Hpsi^1(D) \otimes_{f_\delta} L \cong \Hpg^1(M(\delta))$.
\end{Lemm}
\begin{proof}
For example, see \cite[Section 5]{KPX14}.
\end{proof}

We reduce the proof of our main theorem to the generic case via the next proposition.

\begin{Prop}\label{reducetogeneric}
Assume that Theorem \ref{main} holds for any $L$ and for all of the de Rham generic $\pg$-modules over $\calR_L$.
Then, Theorem \ref{main} holds unconditionally.
\end{Prop}

\begin{proof}
Let $M$ be a de Rham $\pg$-module over $\calR_L$ of rank $r$. 
We prove Theorem \ref{main} for $M$ and $N=\nrig(M)$ by induction on $r$.

The base case $r=1$ has been proved as Theorem \ref{main rank 1}.

Suppose that $r \geq 2$. We assume that Theorem \ref{main} holds for all the de Rham $\pg$-modules over $\calR_L$ of rank $\leqq r-1$. 
If $M$ is not generic, we have $\dcris(M(\delta)) \neq 0$ for some character $\delta: \Gamma \to \overline{\Q}_p^\times$.
Extending $L$ if necessary, we may assume that $\delta(\Gamma) \subseteq L^\times$ and that there is a nonzero $\phi$-eigen vector 
$x\otimes e_{\delta} \in \dcris(M(\delta))=(M(\delta)[1/t])^\Gamma$ with $x\in M[1/t]$.
Then, the submodule $\calR_{L}[1/t] x \subseteq M[1/t]$ is stable under $\pg$-actions.
Since $\calR_L$ is a B\'ezout domain, it turns that out that its saturation $M' \coloneqq \calR_{L}[1/t] x\cap M \subsetneq M(\delta)$ and the quotient $M(\delta)/M'$ are $\pg$-modules.
Therefore, by considering an exact sequence
\[ 0 \to M' \to M(\delta) \to M(\delta)/M'\to 0 \]
of de Rham $\pg$-modules, Lemma \ref{exactness of Exp} gives our assertion.
\end{proof}

\subsection{Proof for generic case}
\label{44}

We continue to use the same notation as in the previous section.

For a technical reason, we introduce another $\pg$-module $M^+$.
Let $h_1 \leq \dots \leq h_r$ be the Hodge-Tate weights of $M$.
Let $\alpha_1, \dots, \alpha_r \in \ddr(M)$ be a basis; taking along the filtration of $\ddr(M)$, we may assume that $t^{h_i} \alpha_i \in \ddif^+(M)$ for each $i$, and that $(t^{h_i}\alpha_i)_{1 \leq i \leq r}$ is a basis of $\ddif^+(M)$.
Then Theorem II.1.2 in \cite{Ber08} gives that there exists a unique $\pg$-module $M^+ \subseteq M$ such that
\[ \mathbf{D}^+_{\dif,n(M)}(M^+) = L_{n(M)}[[t]] \cdot t^{h_1 +1}\alpha_1 \oplus \qty( \bigoplus_{2 \leq i \leq r} L_{n(M)}[[t]] \cdot t^{h_i} \alpha_i ). \] 

Note that, since the big exponential maps are transitive by definition, $t^h N \subseteq M, M'$ for a sufficient large $h \in \Z_{>0}$, and $t^h N$ is obtained by the above procedure repeatedly starting from $M$, it suffices to prove the case $M'=M^+$.
Moreover, by Lemma \ref{Exp and twist universally}, we may assume that $\delta= \mathbf{1}$.

In summary, it is sufficient to prove that the diagram
\[
\xymatrixcolsep{4pc}
\xymatrix{
& 1_L \ar[dl]_-{\varepsilon_{\dR}(M)}\ar[dr]^-{\varepsilon_{\dR}(M^+)}\ar@{}[d]|{}&\\
\Delta_{L}(M)  \ar[rr]_-{ \Exp(M,M^+)_{\mathbf{1}} }&& \Delta_{L}(M^+).
}
\]
commutes.

\begin{Lemm}\label{Second half: Exp_2 and theta_dR}
The diagram
\[ 
\xymatrixcolsep{4pc}
\xymatrix{
\Det_L (\ddr(M)) \ar[r]^-{\times(-1)}  \ar@{<-}[d]_-{f_M}  & \Det_L (\ddr(M^+)) \ar@{<-}[d]^-{f_{M^+}} \\
\Delta_{L,2} (M) \ar[r]_-{\Exp_2(M,M^+)_\mathbf{1}} & \Delta_{L,2} (M^+),
}
\]
commutes.
\end{Lemm}
\begin{proof}

This follows from the direct calculation
\begin{align*}
f_{M^+}(\Exp_2(M,M^+)_\mathbf{1}(x)) &= f_{M^+}(-tx) \\
&= - \frac{1}{\varepsilon(M^+)}\frac{1}{t^{h_{M^+}}} \otimes \phi^n(tx) \\
&= - \frac{1}{\varepsilon(M^+)}\frac{t}{t^{h_{M^+}}} \otimes \phi^n(x) \\
&=- \frac{1}{\varepsilon(M^+)}\frac{1}{t^{h_{M}}} \otimes \phi^n(x) \\
&= - f_M(x),
\end{align*}
where $x \in \mathcal{L}_L(M)$ is any element and $n \geq \max \set{n(M),\,n(M^+)}$.
We note that the last equality follows from the fact that for two de Rham $\pg$-modules $D,D'$ with $D[1/t]=D'[1/t]$, the corresponding filtered $(\phi,\, N,\, G_{\Qp})$-modules are the same, so are the attached $\varepsilon$-constants.
\end{proof}

Thus, the main theorem is deduced from the following lemma.

\begin{Lemm}\label{First half: Exp_1 and theta}
The diagram
\[
\xymatrixcolsep{3pc}
\xymatrix{
& 1_{L}\ar[dl]_-{\Gamma(M) \theta(M)}\ar[dr]^-{\Gamma(M^+) \theta(M^+)}\ar@{}[d]|{}&\\
\Delta_{L,1}(M) \boxtimes_L \Det_L(\ddr(M))  \ar[rr]_-{ -\Exp_1(M,M^+)_\mathbf{1} \otimes_L \id} && \Delta_{L,1}(M^+)\boxtimes_L \Det_L( \ddr(M^+))
}
\]
commutes.

\end{Lemm}

\begin{proof}
By Lemma \ref{reducetogeneric}, we may assume that $M$ is generic, which implies $M^+$ is also generic.
In the following, we use a letter $D$ to denote a general generic de Rham $\pg$-module.
We say $\mathcal{E} = \Exp_1(M,M^+)_\mathbf{1} \otimes_L \id$ for short.

We first give explicit descriptions of the isomorphisms appearing in the diagram.
By Lemma \ref{lemma: generic treats easily}, we have canonical quasi-isomorphisms $C^\bullet_\psi(D) \cong \Hpsi^1(D)[1]$ and $C^\bullet_{\phi,\gamma}(D) \cong \Hpg^1(D)[1]$.
The canonical base change isomorphism
\[ \Delta^\Iw_{L,1}(D) \otimes_{f_\mathbf{1}} L  \isom \Delta_{L,1}(D) \]
is thus the image under $[-1]$-functor of the isomorphism
\[  \Det_\dist \qty(\Hpsi^1(D)) \otimes_{f_1} L \isom \Det_L(\Hsg^1(D)): (\wedge^r x_i) \otimes 1 \mapsto \bigwedge^r \qty[\frac{p-1}{p} \log \chi(\gamma) p_\Delta (x_i),0]. \]
Therefore, the isomorphism $\Exp_1(M,M^+)_\mathbf{1}: \Delta_1(M) \xrightarrow{\sim} \Delta_1(M^+)$ is obtained as the image under $[-1]$-functor of the isomorphism
\[  \Det_L(\Hsg^1(M)) \xrightarrow{\sim} \Det_L(\Hsg^1(M^+)): \wedge^r x_i \mapsto \nabla_{h_1} (\wedge^r x_i) . \]
Next we consider $\theta(D)$.
Under the assumption of genericity, we have $\Hpg^i(D)=0$ for $i=0,2$ and $\dcris(D)=0$ again by Lemma \ref{lemma: generic treats easily}, so $\theta(D)$ is simply obtained via the trivializations of the exact sequences
\[ 0 \to t(D)_3 \xrightarrow{\exp_D} \Hpg^1(D)_{f,4} \to 0 ,\]
\[ 0 \to \Hpg^1(D)_{/f,1} \xrightarrow{\exp^*_D} \ddr^0(D)_2 \to 0,  \]
\[ 0 \to \ddr^0(D)_1 \to \ddr(D)_2 \to t(D)_3 \to 0 ,\]
\[ 0 \to \Hpg^1(D)_{f,1} \to \Hpg^1(D)_2 \to \Hpg^1(D)_{/f,3} \to 0 ,\]
where the index appearing at each space expresses its degree in the sequences and the last two sequences are canonical ones.
More explicitly, $\theta(D): 1_L \xrightarrow{\sim} (\Det_L (\Hsg^1(D)))^{-1} \boxtimes \Det_L (\ddr(D))$ is written as follows: if we put $d_0(D) = \dim_L(\ddr^0(D))$, then for any basis $(\beta_i)_{1 \leq i \leq r} \in \ddr(D)$ such that $(\beta_i)_{d_0(D)+1 \leq i \leq r}$ spans $\ddr^0(D)$, $\theta(D)^{-1}$ is described as
\begin{multline*}
\theta(D)^{-1}:\qty[ \exp_D(\overline{\beta_1})\wedge \dots \wedge \exp_D(\overline{\beta_{r-d_0(D)}}) \wedge \beta^{*_D}_{r-d_0(D)+1} \wedge \dots \wedge \beta^{*_D}_r \mapsto 1  ] \otimes (\wedge^r \beta_i) \\ \mapsto (-1)^{d_0(D)},
\end{multline*}
where $\beta^{*_D}_i$ are any lifts of $\beta_i$ with respect to $\exp^*_D$.

Using the above descriptions, we can say the asserted commutativity in more concrete form.
We put an element $X$ of $(\Det_L (\Hsg^1(D(M))))^{-1} \boxtimes \Det_L(\ddr(M))$ as
\begin{multline*}
\qty[\exp_M(\overline{\alpha_r})\wedge \dots \wedge \exp_M(\overline{\alpha_{d_0+1}}) \wedge \alpha^{*_M}_{1} \wedge \dots \wedge \alpha^{*_M}_{d_0} \mapsto 1] \otimes (\alpha_r \wedge \dots \wedge \alpha_{d_0+1} \wedge \alpha_1 \wedge \dots \wedge \alpha_{d_0}). 
\end{multline*}
Then, since $X$ is a basis by the definition of $(\alpha_i)_{1 \leq i \leq r}$, our claim deduces to show the commutativity at $X$, that is, the equality
\[ \Gamma(M^+)^{-1} \theta(M^+)^{-1} (- \mathcal{E} (X)) = \Gamma(M)^{-1} \theta(M)^{-1}(X), \]
or, furthermore, by the description of $\theta(M)$ above, the equality
\[ - \Gamma(M^+)^{-1} \theta(M^+)^{-1} (\mathcal{E}(X)) = (-1)^{d_0} \Gamma(M)^{-1}. \]

By our construction of $M^+$, we have $\nabla_{h_1}(\ddifm^+(M)) \subseteq \ddifm^+(M^+)$ for all $m \geq n(M)$.
Therefore, we can verify the above equality essentially by Lemma \ref{keylemma} as follows.

For the case $h_1 <0$, Lemma \ref{keylemma} (i) gives that $\nabla_{h_1}(\alpha^{*_M}_1) = -h_1 \alpha^{*_{M^+}}_{1}$,
so one obtains
\begin{multline*}
\mathcal{E}(X)=(-h_1)^{-1} \qty[\exp_{M^+}(\overline{\alpha_r})\wedge \dots \wedge \exp_{M^+}(\overline{\alpha_{d_0+1}}) \wedge \alpha^{*_{M^+}}_{1} \wedge \dots \wedge \alpha^{*_{M^+}}_{d_0} \mapsto 1]\\ \otimes  (\alpha_r \wedge \dots \wedge \alpha_{d_0+1} \wedge \alpha_1 \wedge \dots \wedge \alpha_{d_0}).
\end{multline*}
Since we have $\ddr^0(M^+)=\ddr^0(M)$, by the the description of $\theta(M^+)$ we obtain
\[\theta(M^+)^{-1} (\mathcal{E}(X)) = (-1)^{d_0(M^+)}(-h_1)^{-1} = (-1)^{d_0+1} h^{-1}_1 .\]
Thus, the desired equality is rewritten as
\[  h^{-1}_1 \Gamma(M^+)^{-1} = \Gamma(M)^{-1}, \]
which clearly holds since $\Gamma(D)^{-1}$ is the product of the Hodge-Tate weights of $D$ with multiplicity and by the relation $\Gamma^*(k+1)=k \cdot \Gamma^*(k)$ for any nonzero $k \in \Z $.

The case $h_1 >0$ follows similarly to the previous case $h_1 < 0$, by using Lemma \ref{keylemma} (ii) instead of Lemma \ref{keylemma} (i).

For the last case $h_1 = 0$, canceling $\Gamma(M) = \Gamma(M^+)$ from the equality our assertion becomes the following one:
\[  - \theta(M^+)^{-1} \qty(\mathcal{E}(X)) = (-1)^{d_0}.\]
Lemma \ref{keylemma} (iii) gives that $\nabla_0(\alpha^{*_{M}}_{1}) = \exp_{M^+}(\overline{\alpha_1})$, we obtain
\begin{multline*} 
\mathcal{E}(X) =
\qty[\exp_{M^+}(\overline{\alpha_r})\wedge \dots \wedge \exp_{M^+}(\overline{\alpha_{d_0+1}}) \wedge \exp_{M^+}(\overline{\alpha_1}) \wedge \alpha^{*_{M^+}}_{2} \dots \wedge \alpha^{*_{M^+}}_{d_0} \mapsto 1]\\ \otimes  (\alpha_r \wedge \dots \wedge \alpha_{d_0+1} \wedge \alpha_1 \wedge \dots \wedge \alpha_{d_0}).
\end{multline*}
In this case, the elements $\alpha_2, \dots, \alpha_{d_0}$ spans $\ddr^0(M^+)$, thus we can use the previous explicit description of $\theta(M^+)$ and obtain
\[  \theta(M^+)^{-1} \qty(\mathcal{E}(X)) = (-1)^{d_0(M^+)} = (-1)^{d_0-1},\]
which completes all the cases and finishes the proof.

\end{proof}

\subsection*{Acknowledgements}This work was supported by JSPS 
KAKENHI Grant Number 22K03231.

\end{document}